\documentclass[11pt,reqno]{amsart}
\usepackage[dvips]{epsfig}
\RequirePackage{amsmath}
\usepackage{epsfig}
\usepackage{graphicx,epsfig}
\usepackage{amscd,amsthm,amsfonts,amsopn,amssymb,verbatim}
\usepackage{color}
\usepackage{times}
\numberwithin{equation}{section}
\newtheorem{theorem}{Theorem} 

\newtheorem{corollary}[theorem]{Corollary}
\newtheorem{proposition}[theorem]{Proposition}

\newtheorem{lemma}[theorem]{Lemma}

\theoremstyle{remark}

\def\be{\begin{equation}}
\def\ee{\end{equation}}

\def\ve{\varepsilon}

\allowdisplaybreaks
\def\be{\begin{equation}}
\def\ee{\end{equation}}

\def\ve{\varepsilon}

\allowdisplaybreaks
\begin{document}

\title
[]
{On eigenvalue spacings for the 1-D Anderson model with singular site distribution}
\author
{J.~Bourgain}
\address
{Institute for Advanced Study, Princeton, NJ 08540}
\email
{bourgain@math.ias.edu}
\thanks{This work was partially supported by NSF grant DMS-1301619}

\abstract
We study eigenvalue spacings and local eigenvalue statistics for 1D
lattice Schr\"odinger operators with H\"older regular potential,
obtaining a version of Minami's inequality and Poisson statistics for the local eigenvalue spacings.  The main additional new input
are regularity properties of the Furstenberg measures and the density of states obtained in some of the author's earlier work.
\endabstract
\maketitle

\section{\bf Introduction}

This Note results from a few discussions with A.~Klein (UCI, summer 011) on Minami's inequality and the results
from \cite{G-K} on Poisson local spacing behavior for the eigenvalues of certain Anderson type models.
Recall that the Hamiltonian $H$ on the lattice $\mathbb Z^d$ has the form
\be\label{1.1}
H=\lambda V+\Delta
\ee
with $\Delta$ the nearest neighbor Laplacian on $\mathbb Z^d$ and $V= (v_n)_{n\in\mathbb Z^d}$ IID variables with a
certain distribution.
Given a box $\Omega\subset \mathbb Z^d$, $H_\Omega$ denotes the restriction of $H$ to $\Omega$ with Dirichlet boundary
conditions.
Minami's inequality, which is a refinement of Wegner's estimate, is a bound on the expectation that $H_\Omega$ has two
distinct eigenvalues in a given interval $I\subset \mathbb R$.
This quantity can be expressed as
\be\label{1.2}
\mathbb E\big[Tr\mathcal X_I(H_\Omega)\big(Tr\mathcal X_I(H_\Omega)-1\big)\big]
\ee
where the expectation is taken over the randomness $V$.
An elegant treatment may be found in \cite{C-G-K} (see in particular Theorem 2.1).

Assuming the site distribution has a bounded density, \eqref{1.2} satisfies the expected bound
\be\label{1.3}
C|\Omega|^2 |I|^2.
\ee
More generally, considering a site distribution probability measure $\mu$ which is H\"older with exponent $0<\beta\leq 1$, i.e.
\be\label{1.4}
\mu(I) \leq C|I|^\beta \text { for all intervals $I\subset \mathbb R$}
\ee
it is shown in \cite{C-G-K} that
\be\label{1.5}
\eqref{1.2} \leq C|\Omega|^2 |I|^{2\beta}.
\ee
For the sake of the exposition, we briefly recall the argument.
Rewrite \eqref{1.2} as
\be\label{1.6}
\mathbb E_V\Big[\sum_{j\in\Omega}\langle \delta_j, \mathcal X_I(H^{(V)}_\Omega) \delta_j> \big(Tr \mathcal
X_I(H^{(V)}_\Omega)-1\big)\Big]
\ee
where $(\delta_j)$ denote the unit vectors of $\mathbb Z^d$.
Introduce a second independent copy $W=(w_n)$ of the potential $V$.
Fixing $j\in\Omega$, denote by $(V_j^\bot, \tau_j)$ the potential with assignments $v_n$ for 
$n\not= j$ and $\tau_j$ for $n=j$.
Assuming $\tau_j\geq v_j$, it follows from rank-one perturbation theory that
\be\label{1.7}
Tr\mathcal X_I(H_\Omega^{(V)})\leq Tr\mathcal X_I \big(H_\Omega^{(V_j^\bot, \tau_j)}\big)+1
\ee
and hence
\be\label{1.8}
\eqref{1.6} \leq \mathbb E_V\mathbb E_W \Big[\sum_{j\in\Omega}\langle \delta_j, \mathcal
X_I(H_\Omega^{(V)})\delta_j\rangle \ Tr\mathcal X_I(H_\Omega^{(V_j^\bot, \Vert v_j\Vert _\infty +
w_j)})\Big].
\ee

Next, invoking the fundamental spectral averaging estimate (see \cite{C-G-K}, Appendix A), we have
\be\label{1.9}
\mathbb E_{v_j} [\langle \delta_j, \mathcal X_I(H_\Omega^{(V_j^\bot, v_j)}) \delta_j]\leq C|I|^\beta
\ee 
so that
\be\label{1.10}
\eqref{1.8} \leq C|I|^\beta \sum_{j\in\Omega} \mathbb E_{V_j^\bot}\mathbb E_{w_j}\big[Tr\mathcal
X_I \big(H_\Omega^{(V_j^\bot, \Vert v_j\Vert_\infty +w_j)}\big)\big].
\ee
The terms in \eqref{1.10} may be bounded using a Wegner estimate.
Applying again \eqref{1.9}, the $j$-term in \eqref{1.10} is majorized by $C|\Omega|\,|I|^\beta$, leading
to the estimate $ C|I|^{2\beta} |\Omega|^2$ for \eqref{1.2}.
It turns out that at least in 1D, one can do better than reapplying the spectral averaging estimate.
Indeed, it was shown in \cite{B1} that in 1D, $SO$'s with H\"older regular site distribution have
a smooth density of states.
This suggests in \eqref{1.5} a better $|I|$-dependence, of the form $|I|^{1+\beta}$.
Some additional work will be needed in order to turn the result from \cite{B1} into the required
finite scale estimate.
We prove the following (set $\lambda=1$ in \eqref{1.1}).

\begin{proposition}\label{Proposition1}
Let $H$ be a 1D lattice random $SO$ with H\"older site distribution satisfying \eqref{1.4} for
some $\beta>0$.
Denote $H_N=H_{[1, N]}$.
Then
\be\label{1.11}
\mathbb E[I\cap \text{\,Spec\,} H_N\not=\phi ] \leq C e^{-cN}+CN|I|.
\ee
\end{proposition}

It follows that $\mathbb E[Tr\mathcal X_i(H_N)]\leq Ce^{-cN} +CN^2|I|$.

The above discussion then implies the following Minami-type estimate.  

\begin{corollary}\label{Corollary2}
Under the assumption from Proposition \ref{Proposition1}, we have
\be\label{1.12}
\mathbb E[Tr\mathcal X_I(H_\Omega)(Tr\mathcal X_I(H_\Omega)-1)] \leq C|\Omega|^3 |I|^{1+\beta}
\ee
provided $\Omega\subset\mathbb Z$ is an interval of size $|\Omega|> C_1\log (2+\frac 1{|I|})$,
where $C, C_1$ depend on $V$.
\end{corollary}

Denote $\mathcal N$ the integrated density of states (IDS) of $H$ and $k(E)=\frac{d\mathcal N}{dE}$.
Recall that $k$ is smooth for H\"older regular site distribution (cf. \cite{B1}).

Combined with Anderson localization, Proposition \ref{Proposition1} and Corollary \ref{Corollary2} permit to derive for $H$ as above.

\begin{proposition}\label{Proposition3}

Assuming $\log \frac 1\delta <cN$, we have for $I=[E_0-\delta, E_0+\delta]$ that
\be\label{1.13}
\mathbb E[Tr\mathcal X_I(H_N)]= Nk(E_0)|I| +O\Big(N\delta^2+\delta\log \Big(N+\frac 1\delta\Big)\Big)
\ee
\end{proposition}

and

\begin{proposition}\label{Proposition4}
\be\label{1.14}
\mathbb E[H_\Omega \text { has at least two eigenvalues in } I]\leq C|\Omega|^2|I|^2+C|\Omega|\log^2
\Big(|\Omega|+\frac 1{|I|}\Big).|I|^{1+\beta}.
\ee
\end{proposition}

Following a well-known strategy, Anderson localization permits a decoupling for the contribution of
pairs of eigenvectors with center of localization that are at least $C\log \frac 1{|I|}$-apart.
Invoking \eqref{1.11}, this yields the first term in the r.h.s of \eqref{1.14}.
For the remaining contribution, use Corollary 2.

With Proposition \ref{Proposition3}, \ref{Proposition4} at hand and again exploiting Anderson localization, the
analysis from \cite{G-K} becomes available and we obtain the following universality statement
for 1D random $SO$'s with H\"older regular site distribution.

\begin{proposition}
\label{Proposition5}
Let $E_0\in\mathbb R$ and $I=[E_0, E_0+\frac LN]$ where we let first $N\to\infty$ and then
$L\to\infty$.
The rescaled eigenvalues
$$
\{N(E-E_0)\mathcal X_I(E)\}_{E\in\text{\,Spec\,} H_N}
$$
satisfy Poisson statistics.
\end{proposition}

At the end of the paper, we will make some comments on eigenvalue spacings for the Anderson-Bernoulli
(A-B) model, where in \eqref{1.1} the $v_n$ are $\{0, 1\}$-valued.
Further results in  line of the above for A-B models with certain special couplings $\lambda$ will appear in \cite{B3}.

\section
{\bf Proof of Proposition 1}

Set $\lambda=1$ in \eqref{1.1}. We denote
\be\label{2.1}
M_n=M_n(E)=\prod^1_{j=n} \begin{pmatrix} E-v_j&-1\\ 1&0\end{pmatrix}
\ee
the usual transfer operators.
Thus the equation $H\xi =E\xi$ is equivalent to
\be\label{2.2}
M_n\begin{pmatrix} \xi_1\\ \xi_0\end{pmatrix} = \begin{pmatrix} \xi_{n+1}\\ \xi_n\end{pmatrix}.
\ee

Considering a finite scale $[1, N]$, let $H_{[1, N]}$ be the restriction of $H$ with
Dirichlet boundary conditions.
Fix $I=[E_0-\delta, E_0+\delta]$ and assume $H_{[1, N]}$ has an eigenvalue $E\in I$ with
eigenvector $\xi=(\xi_j)_{1\leq j\leq N}$.
Then
\be\label {2.3}
M_N(E)\begin{pmatrix} \xi_1\\ 0\end{pmatrix} =\begin{pmatrix} 0\\ \xi_N\end{pmatrix}.
\ee
Assume $|\xi_1|\geq |\xi_N|$ (otherwise replace $M_N$ by $M_N^{-1}$ which can be treated similarly).
It follows from \eqref{2.3} that
\be\label{2.4}
\Vert M_N(E) e_1\Vert\leq 1
\ee
with $(e_1, e_2)$ the $\mathbb R^2$-unit vectors.
On the other hand, from the large deviation estimates, we have that
\be\label{2.5}
\log\Vert M_N(E_0)e_1\Vert> cN
\ee
with probability at least $1-e^{-cN}$ (in the sequel, $c, C$ will denote various constants that
may depend on the potential).

Write
\be\label{2.6}
\big|\log\Vert M_N(E) e_1\Vert -\log \Vert M_N(E_0)e_1\Vert\big|\leq \int^\delta_{-\delta} 
\Big| \frac d{dt} [\log\Vert M_N(E_0+t)e_1\Vert]\Big|dt.
\ee
The integrand in \eqref{2.6} is clearly bounded by
\be\label{2.7}
\sum_{j=1, 2} \,  \sum^N_{n=1} \frac {|\langle M_{N-n}^{(v_N, \ldots, v_{n+1})} (E_0+t) e_1, e_j\rangle|.
|\langle M_{n-1}^{(v_{n-1}, \ldots, v_1)} (E_0+t) e_1, e_1\rangle|}
{\Vert M_N^{(v_N, \ldots, v_1)}(E_0+t)e_1\Vert}
\ee
\be\label{2.8}
\leq 2|E-E_0|\sum^N_{n=1} \frac {\Vert M_{N-n}^{(v_N, \ldots, v_{n+1})} (E_0+t)\Vert}
{\Vert M_{N-n}^{(v_N, \ldots, v_{n+1})} (E_0+t) \zeta_n\Vert}
\ee
where
\be\label{2.9}
\zeta_n = \frac {M_{n-1}^{(v_{n-1}, \ldots, v_1)} (E_0+t) e_1}
{\Vert M_{n-1}^{(v_{n-1}, \ldots, v_1)} (E_0+t)e_1\Vert}
\ee
depends only on the variables $v_1, \ldots, v_{n-1}$.

At this point, we invoke some results from \cite{B1}.
It follows from the discussion in \cite{B1}, \S5 on $SO$'s with H\"older potential that for
$\ell> C=C(V)$, the inequality
\be\label{2.10}
\mathbb E_{v_1, \ldots, v_\ell} [\Vert M_\ell(\zeta)\Vert <\ve\Vert M_\ell\Vert]\lesssim \ve
\ee
holds for any $\ve>0$ and unit vector $\zeta \in\mathbb R^2$, $M_\ell = M_\ell^{(v_1, \ldots, v_\ell)}$.

A word of explanation.
It is proved in \cite{B1} that if we take $n$ large enough, the map $(v_1,\ldots, v_n)\mapsto M_n^{(v_n, \ldots, v_n)}$ defines a
bounded density on $SL_2(\mathbb R)$.
Fix then some $n=O(1)$ with the above property and write for $\ell>n$,
$$
\Vert M_\ell(\zeta)\Vert \geq |\langle M_n(\zeta), M^*_{\ell-n} e_j\rangle| \qquad (j= 1, 2)
$$
noting that here $M_n$ and $M_{\ell-n}$ are independent as functions of the potential.
Choose $j$ such that $\Vert M^*_{\ell-n} e_j\Vert\sim\Vert M^*_{\ell-n}\Vert = M_{\ell-n}\Vert\sim \Vert M_\ell\Vert
$ and fix the vector $M^*_{\ell-n} e_j$.
Since then $(v_1, \ldots, v_n)\mapsto M_n(\zeta)$ defines a bounded density, inequality \eqref{2.10} holds.

Since always $\Vert M_\ell\Vert < C^\ell$ and $\Vert M_\ell(\zeta)\Vert> C^{-\ell}$, it clearly
follows from \eqref{2.10} that
\be\label{2.11}
\mathbb E_V\Big[\frac{\Vert M_\ell^{(V)} \Vert}{\Vert M_\ell^{(V)} (\zeta)\Vert}\Big]
\leq C\ell.
\ee
Therefore
\be\label{2.12}
\mathbb E_V[\eqref{2.8}]< CN^2\delta.
\ee
Hence, we showed that, assuming \eqref{2.5}, Spec$\, H_N^{(V)}\cap I\not=\phi$ with
probability at most $CN\delta$.
Therefore Spec$\,H_N^{(V)} \cap I\not= \phi$ with probability at most $CN\delta+ Ce^{-cN}$, proving
\eqref{1.11}.

\section
{\bf Proof of Propositions \ref{Proposition3} and \ref{Proposition4}}

Assume $\log \frac 1{|I|}< cN$ and set $M=C\log \big(N+\frac 1{|I|}\big)$ for appropriate constants
$c, C$.
From the theory of Anderson localization in 1D, the eigenvectors $\xi_\alpha$ of $H_N$, $|\xi_\alpha|=1$
satisfy
\be\label{3.1}
|\xi_\alpha (j)|< e^{-c|j-j_\alpha|} \text { for } |j-j_\alpha|>\frac M{10}
\ee
with probability at least $1-e^{-cM}$, with $j_\alpha$ the center of localization of $\xi_\alpha$.

The above statement is well-known and relies on the large deviation estimates for the transfer matrix.
Let us also point out however that the above (optimal) choice of $M$ is not really important in what follows and
taking for $M$ some power of the log would do as well.

We may therefore introduce a collection of intervals $(\Lambda_s)_{1\leq s\lesssim \frac NM}$ of
size $M$ covering $[1, N]$, such that for each $\alpha$, there is some $1\leq s\lesssim \frac NM$
satisfying
\be\label{3.2}
j_\alpha\in\Lambda_s \text { and } \Vert\xi_\alpha|_{[1, N]\backslash \Lambda_s}\Vert<
e^{-cM}
\ee
\be\label{3.3}
\Vert (H_{\Lambda_s} -E_\alpha) \xi_{\alpha, s}\Vert< e^{-cM}
\ee
with $\xi_{\alpha, s} =\xi_\alpha|\Lambda_s$.
Therefore dist$\,(E_\alpha$, Spec$\,H_{\Lambda_s})< e^{-cM}<\delta$.

Let us establish Proposition \ref{Proposition3}.
Denoting $\Lambda_1$ and $\Lambda_{s_*}$ the intervals appearing at the boundary of $[1, N]$, one obtains by a well-known argument that
\be\label{3.4}
\mathbb E[Tr \mathcal X_I(H_N)]= N.\mathcal N(I)+O\big(e^{-cM}+\mathbb E[Tr \mathcal X_{\tilde I}(H_{\Lambda_1})]+\mathbb E [Tr 
\mathcal X_{\tilde I}(H_{\Lambda_{s_*}})]\big)
\ee
with $\tilde I=[E_0-2\delta, E_0+2\delta]$.
Invoking then Proposition \ref{Proposition1} and Corollary \ref{Corollary2}, we obtain
\be\label{3.5}
\mathbb E[Tr\mathcal X_I(H_{\Lambda_s})] < ce^{-cM}+CM\delta+CM^3 \delta^{1+\beta}< CM\delta
\ee
by the choice of $M$ and assuming $(\log N)^2 \delta^\beta<1$, as we may.

Substituting \eqref{3.5} in \eqref{3.4} gives then
$$
\begin{aligned}
&N\int_I k(E)dE+O(M\delta)=\\
&N k(E_0)|I|+O\Big(N\delta^2+\delta\log \Big(N+\frac 1\delta\Big)\Big)
\end{aligned}
$$
since $k$ is Lipschitz.
This proves \eqref{1.13}.

Next, we prove Proposition \ref{Proposition4}.

Assume $E_\alpha, E_{\alpha'} \in I, \alpha\not=\alpha'$.
We distinguish two cases.
\smallskip

\noindent
{\bf Case 1.} $|j_\alpha -j_{\alpha'}|>CM$.

Here $C$ is taken large enough as to ensure that the corresponding boxes $\Lambda_s, \Lambda_{s'}$
introduced above are disjoint. Thus
\be\label{3.6}
\text{Spec\,} H_{\Lambda_s}\cap I\not=\phi
\ee
\be\label{3.7}
\text{Spec\,} H_{\Lambda_{s'}} \cap I\not= \phi.
\ee
Since the events \eqref{3.6}, \eqref{3.7} are independent, it follows from Proposition
\ref{Proposition1} that the
probability for the joint event is at most
\be\label{3.8}
C e^{-cM}+CM^2 \delta^2<CM^2\delta^2
\ee
by our choice of $M$.
Summing over the pairs $s, s'\lesssim \frac NM$ gives therefore the bound $CN^2\delta^2$ for
the probability of a Case 1 event.
\medskip

\noindent
{\bf Case 2.} $|j_\alpha -j_{\alpha'}|\leq CM$.

We obtain an interval $\Lambda$ as union of at most $C$ consecutive $\Lambda_s$-intervals such that
\eqref{3.2}, \eqref{3.3} hold with $\Lambda_s$ replaced by $\Lambda$ for both $(\xi_\alpha,
E_\alpha)$, $(\xi_{\alpha'}, E_{\alpha'})$.
This implies that Spec$\,H_\Lambda\cap \tilde I$ contains at least two elements.
By Corollary \ref{Corollary2}, the probability for this is at most $CM^3\delta^{1+\beta}$.
Hence, we obtain the bound $CM^2N\delta^{1+\beta}$ for the Case 2 event.

The final estimate is therefore
$$
e^{-cM}+CN^2\delta^2 +CM^2N\delta^{1+\beta}
$$
and \eqref{1.14} follows from our choice of $M$.

\section
{\bf Sketch of the proof of Proposition \ref{Proposition5}}

Next we briefly discuss local eigenvalue statistics, following \cite{G-K}.

The Wegner and Minami type estimates obtained in Proposition 3 and 4  above permit to reproduce
essentially the        analysis from [G-K] proving local Poisson statistics
for the eigenvalues of $H^\omega_N$.
We sketch the details (recall that we consider a 1D model with H\"older site distribution).

Let $M=K \log N, M_1=K_1\log N$ with $K\gg K_1\gg 1$ ($\rightarrow \infty$ with $N$) and partition
$$
\Lambda=[1, N] =\Lambda_1\cup \Lambda_{1, 1}\cup\Lambda_2 \cup\Lambda_{2, 1}\cup \ldots
=                      \bigcup_{\alpha\lesssim \frac N{M+M_1}}
(\Lambda_\alpha\cup\Lambda_{\alpha, 1})
$$
where $\Lambda_\alpha$ (resp. $\Lambda_{\alpha, 1}$) are $M$ (resp. $M_1$) intervals

\includegraphics[width=3.6 in]{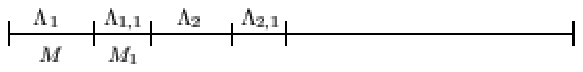} 

Denote
$$
\begin{matrix}
&\mathcal E_\alpha \ \ = \text { eigenvalue of $H_\Lambda$ with center of localization in } &
\Lambda_\alpha
\\
&\mathcal E_{\alpha, 1} = \qquad \qquad \qquad\hrulefill  \qquad\qquad\qquad  & \ \Lambda_{\alpha, 1}
\end{matrix}
$$
Let $\Lambda_\alpha' $ (resp. $\Lambda_{\alpha, 1}'$) be a neighborhood of $\Lambda_\alpha$ (resp.
$\Lambda_{\alpha,  1}$) of size $\sim \log N$ taken such as to ensure that
$$
\text{dist\,} (E, \text { Spec\,} H_{\Lambda_\alpha'} ) <\frac 1{N^A} \text { for $E\in\mathcal
E_\alpha$}
$$
($A$ a sufficiently large constant), and
\be\label{4.1}
\text{dist\,} (E, \text { Spec\,}H_{\Lambda_{\alpha, 1}'})< \frac 1{N^A} \text { for $E\in\mathcal
E_{\alpha, 1}$}.
\ee
Choosing $K_1$ large enough, we ensure that the $\Lambda_\alpha'$ are disjoint and hence
$\{\text{Spec\,}          H^\omega_{\Lambda_\alpha'}\}$ are independent.

Consider an energy interval
$$
I=\Big[E_0, E_0+\frac LN\Big]
$$
Denote
$$
P_\Omega(I) =\mathcal X_I(H_\Omega)
$$
with $L$ a large parameter, eventually  $\rightarrow\infty$.

We obtain from \eqref{1.11} and our choice of $M_1$ that
$$
\mathbb P[\mathcal E_{\alpha, 1} \cap I \not=\phi]  \lesssim   M_1|I|
$$
and hence
\be\label{4.2}
\mathbb P[\bigcup_\alpha \mathcal E_{\alpha, 1} \cap I\not= \phi]\lesssim \frac NM M_1|I| \lesssim \frac{LK_1}K =
o(1)
\ee
provided
\be\label{4.3}
K_1L =o(K).
\ee
Also, by \eqref{1.12}
\begin{align}\label{4.4}
&\mathbb P [|\mathcal E_\alpha\cap I|\geq 2]\leq\nonumber\\
&\mathbb P[H_{\Lambda_\alpha'} \text { has at least two eigenvalues in } \tilde I]\lesssim M^3
|I|^{1+\beta}<   M^3 \frac {L^{1+\beta}}{N^{1+\beta}}
\end{align}
so that
\be\label{4.5}
\mathbb P[\max_\alpha |\mathcal E_\alpha\cap I|\geq 2] \lesssim \frac NM (4.4) \lesssim \frac
{M^2L^{1+\beta}}{N^\beta} < N^{-\beta/2}.
\ee
Next, we introduce  the (partially defined) random variables
\be\label{4.6}
E_\alpha(V) =\sum_{E\in\text { Spec\,} H_{\Lambda_\alpha'} } E \, 1_I(E) \text { provided
$|$Spec      $H_{\Lambda_\alpha'}\cap I|\leq 1$}.
\ee
Thus the $E_\alpha, \alpha =1, \ldots, \frac N{M+M_1}$ take values in $I$, are independent
and have the same distribution.

Let $J\subset I$ be an interval, $|J|$ of the order of $\frac 1N$.
Then by \eqref{4.4} and Proposition \ref{Proposition3}.
\be\label{4.7}
\mathbb E[1_J(E_\alpha)]= \mathbb E [Tr\, P_{\Lambda_\alpha'} (J)]+ O\Big(\frac
1{N^{1+\beta/2}}\Big) =k(E_0)\Big(1+O\Big(\frac 1K\Big)\Big) |J|M'
\ee
where $M' =|\Lambda_\alpha'|$.

Therefore $\{N(E_\alpha -E_0)1_I(E_\alpha)\} _{\alpha\leq \frac N{M+M_1}}$ satisfies Poisson statistics (in a
weak sense), proving Proposition \ref{Proposition5}.

\section
{\bf Comments on the Bernoulli case}

Consider the model \eqref{1.1} with $V=(v_n)_{n\in \mathbb Z}$ independent $\{0, 1\}$-valued.
For large $|\lambda|$, $H$ does not have a bounded density of states.
It was shown in \cite{B2} that for certain small algebraic values of the coupling constant $\lambda$,
$k(E)=\frac {d\mathcal N}{dE}$ can be made arbitralily smooth (see \cite{B2} for the precise statement).
In particular $k\in L^\infty$ and one could ask if Proposition \ref{Proposition4} remains valid in this
situation.
One could actually conjecture that the analogue of Proposition \ref{Proposition4} holds for the A-B model
in 1D, at small disorder.
This problem will be pursued further in \cite{B3}.
What we prove here is an eigenvalue separation property at finite scale for the A-B model at arbitrary disorder
$\lambda\not=0$.
Denote again $H_N$ the restriction of $H$ to $[1, N]$ with Dirichlet boundary conditions. We have

\begin{proposition}\label{Proposition6}
With large probability, the eigenvalues of $H_N$ are at least $N^{-C}$ separated, $C=C(\lambda)$.
\end{proposition}

A statement of this kind is known for random $SO$'s with H\"older site distribution of regularity $\beta>\frac 12$, in arbitrary dimension.
But note that our proof of Proposition \ref{Proposition6} is specifically 1D, as will be clear below.
There are three ingredients, each well-known.

\noindent
{\bf 1. Anderson localization}

Anderson localization holds also for the 1D A-B model at any disorder.
In fact, there is the following quantitative form.
Denote $\xi^{(1)}, \ldots, \xi^{(N)}$ the normalized eigenvectors of $H_N$.
Then, with large probability $(> 1-N^{-A})$, each $\xi^{(j)}$ is essentially localized on some interval of size $C(\lambda)\log N$, in the
sense that there is a center of localization $\nu_j\in [1, N]$ such that
\be\label{5.1}
|\xi_n^{(j)}|< e^{-c(\lambda)|n-\nu_j|} \text { for } |n-\nu_j|> C(\lambda) \log N.
\ee
\medskip

\noindent
{\bf 2. H\"older regularity of the IDS}

The IDS $\mathcal N(E)$ of $H$ is H\"older of exponent $\gamma=\gamma(\lambda)> 0$.
There are various proofs of this fact (see in particular \cite{C-K-M} and \cite{S-V-W}).
In fact, it was shown in \cite{B1} that $\gamma(\lambda)\to 1$ for $\lambda\to 0$ but we will not
need this here.
What we use is the following finite scale consequence.

\begin{lemma}\label{Lemma6}
Let $M\in\mathbb Z_+$, $E\in\mathbb R$, $\delta>0$.  Then
\begin{align}\label{5.2}
&\mathbb E[\text{there is a vector $\xi =(\xi_j)_{1\leq j\leq M}, \Vert\xi\Vert=1$, such that}
\nonumber\\
&\Vert(H_M-E)\xi\Vert<\delta, |\xi_1|<\delta, |\xi_M|<\delta] \leq CM\delta^\gamma.
\end{align}
\end{lemma}

The derivation is standard and we do  briefly recall the argument.

Take $N\to\infty$ and split $[1, N]$ in intervals of size $M$. Denoting $\tau$ the l.h.s. of \eqref{5.2}, we see that
$$
\mathbb E\big[\#(\text{Spec\,} H_N\cap [E-5\delta, E+5\delta])\big] \geq \frac NM\tau.
$$
Dividing both sides by $N$ and letting $N\to\infty$, one obtains that
$$
\frac \tau M\leq \mathcal N([E-5\delta, E+5\delta])
$$
where $\mathcal N$ is the IDS of $H$. 

\medskip

\noindent
{\bf 3. A repulsion phenomenon}

The next statement shows that eigenvectors with eigenvalues that are close together have their
centers far away.
The argument is based on the transfer matrix and hence strictly 1D.

\begin{lemma}\label{Lemma7}
Let $\xi, \xi'$ be distinct normalized eigenvectors of $H_N$ with centers $\nu, \nu'$,
\begin{align}\label{5.3}
H_N\xi&= E\xi\nonumber\\
H_N\xi'&=E'\xi.
\end{align}
Assuming $|E-E'|<N^{-C(\lambda)}$, it follows that
\be\label{5.4}
|\nu-\nu'|\gtrsim \log \frac 1{|E-E'|}.
\ee
\end{lemma}

\begin{proof}
Let $\delta =|E-E'|$ and assume $1\leq \nu \leq \nu'\leq N$.
Take $M=C(\lambda)\log^N$ satisfying \eqref{5.1} and $\Lambda$ an $M$-neighborhood of 
$[\nu, \nu']$ in $[1, N]$.

In particular, we ensure that
\be\label{5.5}
|\xi_n|, |\xi_n'|< N^{-10} \text { for } n\not\in \Lambda.
\ee
We can assume that $|\xi_\nu|>\frac 1{2\sqrt M}$.
Since $\Vert\xi_\nu' \xi - \xi_\nu
\xi'\Vert\geq  |\xi_\nu|> \frac 1{2\sqrt M}$,
it follows from \eqref{5.5} that for some $n_0\in\Lambda$
\be\label{5.6}
|\xi_\nu'\xi_{n_0} -\xi_\nu\xi_{n_0}'|\gtrsim \frac 1{\sqrt M\sqrt{|\Lambda|}}.
\ee
Next, denote for $n\in [1, N]$
$$
D_n=\xi_\nu'\xi_n-\xi_\nu \xi_n'
$$
and
$$
W_n=\xi_n' \xi_{n+1} -\xi_n\xi_{n+1}'.
$$
Clearly, using the equations \eqref{5.3}
\be\label{5.7}
\Vert(H_N-E) D\Vert \leq \delta
\ee\and
\be\label{5.8}
\sum_{1\leq n<N} |W_n-W_{n+1}|<\delta.
\ee
Let $\nu <N$.
Since $D_\nu=0$, it follows from \eqref{5.7} that
\be\label{5.9}
|D_n|\leq (2+|\lambda|+|E|)^{|n-\nu|} (|D_{\nu+1}| +2\delta).
\ee
(If $\nu =N$, replace $\nu+1$ by $\nu-1$).
From \eqref{5.6}, \eqref{5.9}
$$
\frac 1{\sqrt M\sqrt{|\Lambda|}} \lesssim (2+|\lambda| +|E|)^{|\Lambda|} (|D_{\nu+1} +2\delta)
$$
and since $D_{\nu+1}=W_\nu$, it follows that
\be\label{5.10}
|W_\nu|+2\delta> 10^{-|\Lambda|}.
\ee
Invoking \eqref{5.8}, we obtain for $n\in [1, N]$
\be\label{5.11}
|W_n|>10^{-|\Lambda|} -(|n-\nu|+1)\delta.
\ee
On the other hand, by \eqref{5.1}
$$
|W_n|\leq |\xi_n|+|\xi_{n+1}|< e^{-c\lambda^2|n-\nu|} \text { for } |n-\nu|>C(\lambda)\log N.
$$
Taking $|n-\nu|\sim |\Lambda|$ appropriately, it follows that
$$
\delta\gtrsim \frac 1{|\Lambda|} 10^{-|\Lambda|}
$$
and hence
$$
|\nu-\nu'|+M\gtrsim \log\frac 1\delta.
$$
Lemma \ref{Lemma7} follows.
\end{proof}
\medskip

\noindent
{\bf Proof of Proposition \ref{Proposition6}.}

Assume $H_N$ has two eigenvalues $E, E'$ such that
$$
|E-E'|<\delta< N^{-C_1}
$$
where $C_1$ is the constant from Lemma \ref{Lemma7}.
It follows that the corresponding eigenvectors $\xi, \xi'$ have resp. centers $\nu, \nu'\in[1, N]$ satisfying
\be\label{5.12}
|\nu-\nu'|\gtrsim \log \frac 1\delta.
\ee
Introduce $\delta_0>\delta$ (to specify), $M=C_2(\lambda) \log \frac 1{\delta_0}$ and $\Lambda=[\nu-M, \nu+M]\cap [1, N]$,
$\Lambda'= [\nu'-M, \nu'+M]\cap [1, N]$.
Let $\tilde\xi =\frac {\xi|_\Lambda}{\Vert\xi|_\Lambda\Vert},\tilde\xi'=\frac {\xi'|_{\Lambda'}} {\Vert\xi'|_{\Lambda'}\Vert}$.
According to \eqref{5.1}, choose $M$ such that
\be\label{5.13}
\Vert (H_\Lambda -E)\tilde\xi\Vert < e^{-c\lambda^2M}<\delta_0 \text { and } |\xi|_{\partial\Lambda}|<\delta_0
\ee
and
\be\label{5.14}
\Vert H_{\Lambda'} -E')\tilde \xi'\Vert<\delta_0 \text { and }  |\xi'|_{\partial\Lambda '}|<\delta_0.
\ee
Requiring
$$
\log\frac 1\delta>C_3M
$$
\eqref{5.12} will ensure disjointness of $\Lambda, \Lambda'$.
Hence $H_\Lambda, H_{\Lambda'}$ are independent as functions of $V$.
It follows in particular from \eqref{5.13} that dist$\,(E, \text{Spec\,} H_\Lambda)<\delta_0$, hence $|E-E_0|<\delta_0$ for some
$E_0\in\text{Spec\,} H_\Lambda$.
Having fixed $E_0$, \eqref{5.14} implies that
\be\label{5.15}
\Vert (H_{\Lambda'}-E_0)\xi'\Vert < |E-E'|+2\delta_0 <3\delta_0.
\ee
Apply Lemma \ref{Lemma6} to $H_{\Lambda'}$ in order to deduce that the probability for
\eqref {5.15} to hold with $E_0 \in \text {Spec\,} H_\Lambda$ fixed, is at most $CM\delta_0^\gamma$.
Summing over all $E_0\in\text {\, Spec\,} H_\Lambda$ and then over all pairs of boxes $\Lambda, \Lambda'$
gives the bound
\be\label{5.16}
O(N^2M^2\delta_0^\gamma)= O\Big(N^2\Big(\log \frac 1{\delta_0}\Big)^2 \delta_0^\gamma\Big)
<N^2 \delta_0^{\gamma/2}.
\ee
It remains to take $\delta_0 =N^{-\frac 5\gamma}$, $\log \frac 1\delta> C\log \frac 1{\delta_0}$.
\medskip

\noindent
{\bf Acknowledgment.} The author is grateful to A.~Klein for his comments and to the UC Berkeley mathematics department
for their hospitality.

\end{document}